\documentclass{amsart}
\usepackage{amssymb,latexsym, amsmath, amscd}
\usepackage{enumerate}
\usepackage{fullpage, booktabs}
\usepackage{graphicx}
\usepackage[all]{xy}
\makeatletter
\@namedef{subjclassname@2010}{%
\textup{2010} Mathematics Subject Classification}
\makeatother

\newtheorem{thm}{Theorem}
\newtheorem{cor}[thm]{Corollary}
\newtheorem{lem}[thm]{Lemma}

\theoremstyle{definition}

\newtheorem{rem}[thm]{Remark}

\numberwithin{equation}{section}

\def\Z{\mathbb Z}

\def\R{\mathbb R}

\def\PB{\overline{{B}}}
\def\PP{\overline{{P}}}
\def\PBgr{\overline{\mathcal{B}}}
\def\Jgr{{\mathcal{J}}}
\def\PBgrness{\overline{\mathcal{N}}}

\begin{document}

\title{A presentation for the planar pure braid group}

\author[J. Mostovoy]{Jacob Mostovoy}

\address{Departamento de Matem\'aticas, CINVESTAV-IPN\\ Col. San Pedro Zacatenco, M\'exico, D.F., C.P.\ 07360\\ Mexico}
\email{jacob@math.cinvestav.mx}

\begin{abstract}
We describe the construction of a minimal presentation for the group of planar pure braids $\PP_n$ on $n$ strands. The generators of this presentation are dual to the generators of the cohomology ring of $\PP_n$ found by Y.~Baryshnikov while the relations lie in the commutator subgroup of $\PP_n$. We also construct an explicit (although non-minimal) presentation for the pure cactus group.
\end{abstract}

\subjclass[2010]{14N20, 20F05}


\maketitle

\section{Introduction}

\subsection{Planar braids}
The group of planar braids on $n$ strands $\PB_n$, also called the \emph{twin group} \cite{Kh2} or the group of \emph{flat briads} \cite{Merk2}, has a presentation with the generators by $\sigma_1, \ldots, \sigma_{n-1}$ and the relations
\[
\begin{array}{rcll}
\sigma_i^2&=& 1&\quad \text{for all\ } 1\leq i < n;\\
\sigma_i\sigma_j&=& \sigma_j\sigma_i &\quad \text{for all\ } 1\leq i, j <n \text{\ with\  } |i-j|>1.
\end{array}
\]
There is a homomorphism of $\PB_n$ onto the symmetric group $S_n$ which sends the generator $\sigma_i$ to the transposition $\tau_i = (i\  i+1)$. The kernel of this homomorphism is the \emph{planar pure braid group} $\PP_n$.  It is the fundamental group of the 
configuration space of $n$ ordered particles in $\R$ no three of which are allowed to coincide. It has been studied by many authors; see, for instance, \cite{BSV, B, BW, GLR, Kh2, HK, Merk2, MRM, NNS, V}.

The planar pure braid groups on 3, 4 or 5 strands are free on 1, 7 and 31 generators respectively, while  the planar pure braid group on 6 strands is a free product of 71 copies of the infinite cyclic group and 20 copies of the free abelian group on 2 generators, see \cite{MRM}. In the present note, we give an algorithm for producing a presentation for $\PP_n$ with arbitrary $n$ with the minimal number of generators and relations.  The cohomology ring of $\PP_n$  was computed by Baryshnikov in \cite{B}, see also \cite{DT}. Our set of generators is dual to the basis in $H^1(\PP_n,\Z)$ constructed by Baryshnikov, while the relations correspond to his basis of $H^2(\PP_n,\Z)$.

We identify the set of strands of a braid in $\PP_n$ with $\{1,\ldots, n\}$, numbering the strands from left to right; the braids are assumed to be ``descending''. A permutation $\alpha$ of the set $\{1,\ldots, n\}$ will be usually encoded by the sequence of numbers $(\alpha^{-1}(1), \ldots, \alpha^{-1}(n))$.

\subsection{Crossing types}
A \emph{crossing type on $n$ strands} $I=(I_1, I_2, I_3)$ is an ordered partition 
$$\{1,\ldots, n\}=I_1\sqcup I_2\sqcup I_3$$ 
with $|I_2|=2$.
Similarly, a \emph{double crossing type on $n$ strands} $J=(J_1, J_2, J_3, J_4, J_5)$    is an ordered partition 
$$\{1,\ldots, n\}=J_1\sqcup J_2\sqcup J_3\sqcup J_4\sqcup J_5$$ 
with $|J_2| = |J_4| =2$. A double crossing type $J$
gives rise to a pair of crossing types 
$$J_{[1]}:=(J_1, J_2, J_3\sqcup J_4\sqcup J_5)\quad \text{and}\quad J_{[2]}:=(J_1\sqcup J_2\sqcup J_3,  J_4, J_5).$$ 
A crossing type $I$ is \emph{essential} if there exists $x\in I_3$ with $x>y$ for each $y\in I_2$. Similarly, a double crossing type $J$ is essential if (a) if there exists $x_1\in J_3$ with $x_1>y$ for each $y\in J_2$ and (b) if there exists $x_2\in J_5$ with $x_2>y$ for each $y\in J_4$. 

\begin{figure}[ht]
\includegraphics[height=1in]{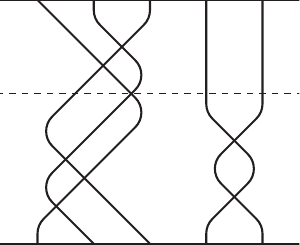}
\caption{A  crossing of a planar pure braid with the crossing type $(\{3\},\{1,2\},\{4,5\})$}
\end{figure}
The term ``crossing type'' has the following explanation. Consider a planar pure braid whose all crossing points have different heights. Then, each crossing of such a braid determines a crossing type $I$: the sets $I_1$ and $I_3$ consist of the strands that pass to the left and to the right of the crossing, respectively, and $I_2$ consists of the strands involved in the crossing. 
A double crossing type $J$ describes,  in a similar fashion, the combinatorics of a pair of crossings that lie on the same level; the two crossing types $J_{[1]}$ and $J_{[2]}$ that it produces appear as the crossing types of a small perturbation of $J$ that separates the heights of the two crossing points. In the terminology of \cite{B}, crossing types and double crossing types are particular cases of \emph{tidy posets}.

\subsection{A presentation for $\PP_n$}

A fixed crossing type $I$ determines a  one-dimensional cohomology class of $\PP_n$, namely, the function that assigns to a braid 
the number of its crossings of type $I$. The crossings should be counted with appropriate signs: a crossing is positive if the strand with the smaller index ``overtakes'' the strand with the larger index as the parameter along the braid increases, and negative otherwise.
The integral cohomology ring of $\PP_n$ was computed by Baryshnikov \cite{B, DT}; in degrees one and two his result\footnote{The Betti numbers of $\PP_n$ had been computed before Baryshnikov by Bjorner and Welker in \cite{BW}.} can be stated as follows:

\begin{thm} The abelian group $H^1(\PP_n,\Z)$ is freely generated by the essential crossing types on $n$ strands and $H^2(\PP_n,\Z)$ -- by the essential double crossing types on $n$ strands.
\end{thm}

In this note, we construct a presentation for $\PP_n$ with one generator for each essential crossing type $I$ on $n$ strands and one relation for each essential double crossing type $J$ on $n$ strands:
 
\begin{thm}\label{presentation}
The group $\PP_n$ has a presentation whose generators are the essential crossing types on $n$ strands, with one relation $r_J$ for each essential double crossing type $J$ on $n$ strands. For every such $J$ there exist words $u_J, v_J$ which do not contain $J_{[1]}$ and $J_{[2]}$ and
such that the relation $r_J$ is of the form
$$[u_J\, J_{[1]}\, v_J^{-1}, J_{[2]}]=1.$$
\end{thm}

The words $u_J$ and $v_J$ are constructed in terms of spanning trees for the 1-skeleta of the permutohedra. Although we do not have a closed form for these words in general, we describe an algorithm that produces them (see Section~\ref{algo}). In the case of at most 5 strands the essential crossing types form a set of free generators for $\PP_n$. When $n=6$, the presentation of this theorem can be simplified so that the relations become commutators of generators. 
In general, finding a closed presentation for $\PP_n$ is not expected to be easy \cite{NNS}.

It will be clear from the construction that Baryshnikov's generators for $H^1(\PP_n,\Z)$ are dual to those of Theorem~\ref{presentation}. As the number of generators is bounded below by the rank of $H^1(\PP_n,\Z)$, we see that our presentation has the smallest possible number of generators. The number of relations is also minimal since it is bounded  below by the rank of $H^2(\PP_n,\Z)$ (see, for instance, \cite[Section~5.4.1]{FM}). 

Our main tool is the Reidemeister-Schreier method. This method has an \emph{ad hoc} component which is finding the right choice for the coset representatives. In the case of the planar pure braid groups, the problem of finding suitable representatives  may be challenging already for $n=5$, see \cite{BSV}. Our choice of representatives (which, in our approach, translates into the choice of a spanning tree of a certain groupoid) is informed by Baryshnikov's work.

Another approach to the presentation of $\PP_n$ has been pointed out by D.~Farley \cite{F} after the first version of the present paper appeared on arxiv. Namely, the groups $\PP_n$ are diagram groups in the sense of \cite{GS1}, and it follows that one can also arrive to a minimal presentation of $\PP_n$ by applying the results of Guba and Sapir \cite[Theorem~6.6]{GS2}. It seems, however, that this approach does not offer substantial computational advantages over the one presented here.

One advantage of using the Reidemeister-Schreier method as applied in the present note is that it can be used to obtain explicit (although not minimal) simple presentations for other pure braid-like groups. We illustrate this on the example of the pure cactus group $\Gamma_n$. The problem of finding an explicit presentation for this group has been raised in the literature  \cite{Gen}, and the Reidemeister-Schreier method has been applied in this context before; see the computations for $n=4,5$ in \cite{BCL}. Here, the real problem that the Reidemeister-Schreier method does not solve is finding a presentation with the minimal number of generators. 

\section{The Reidemeister-Schreier presentations}\label{RSp}

\subsection{The Reidemeister-Schreier method via groupoids}\label{rs}
We will use the Reidemeister-Schreier method as formulated by Higgins \cite[Theorem~7]{Higgins}.

Let $\mathcal{C}=\langle X\,|\, R\rangle$ be a connected groupoid presented in terms of generators and relations and $T$ be a set of words in $X$ such that the corresponding morphisms in $\mathcal{C}$ form a spanning tree for $\mathcal{C}$. Here we think of a groupoid as a directed graph with loops and multiple edges; a spanning tree of the groupoid is a spanning tree for this graph (where we disregard the directions of the edges).

\begin{thm}\label{Higgins}
The group of automorphisms of an object in $\mathcal{C}$ has a presentation $\langle X\,|\, R \cup T\rangle$.
\end{thm}

For $a\in\mathrm{Ob}(\mathcal{C})$, an explicit isomorphism 
$$\langle X\,|\, R\cup T\rangle \to \mathrm{Aut}(a)$$ 
can be constructed as follows. 
For each object $a'$, let $t(a')\in \mathrm{Hom} (a, a')$ be the morphism obtained by following the chosen spanning tree from $a$ to $a'$. (If an edge $m$ in the path from $a$ to $a'$ has the orientation opposite to the direction of the path, we replace it with $m^{-1}$). Assume that the morphism $m_x$ in $\mathcal{C}$ which corresponds to $x\in X$ lies in $\mathrm{Hom} (a_1, a_2)$. Then, $x\in X$ gives rise to the element 
\begin{equation}\label{braidgen}
t(a_1)\cdot m_x \cdot t^{-1}(a_2)\in \mathrm{Aut}(a).
\end{equation}

The usual Reidemeister-Schreier method for constructing a presentation of a subgroup $H$ in a group $G$ with a known presentation becomes a particular case of this result if one considers the action groupoid of $G$ on the set of cosets $G/H$.

\subsection{The groupoid of planar braids}
Consider the groupoid $\PBgr_n$ whose objects are the permutations of the set $\{1,\ldots, n\}$ and the morphisms are the planar braids whose strands are labelled by the numbers from 1 to $n$. Such a braid defines two permutations of $\{1,\ldots, n\}$; namely, the order of the labels at the top and at the bottom of the braid. These two permutations are the source and the target of the morphism defined by the braid.

Two braids with labelled strands are composable if and only if the labels of the strands match at the corresponding endpoints; their composition, if defined, is the product of the braids. 
The automorphism group of any permutation is isomorphic to the planar pure braid group. 

\begin{figure}[ht]
\includegraphics[height=0.8in]{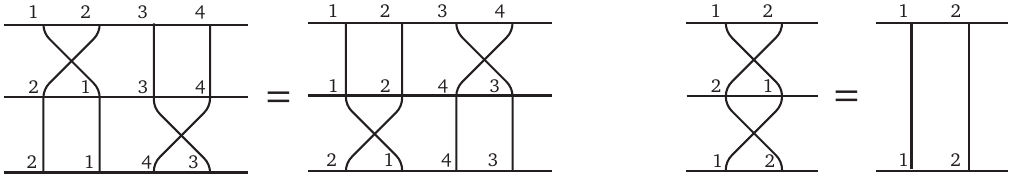}
\caption{Relations in the groupoid of planar braids}
\end{figure}

The groupoid $\PBgr_n$ has a set of generators $[\alpha,k]\in\mathrm{Hom} (\alpha, \alpha\tau_k),$ where $\alpha$ is a permutation of $\{1,\ldots, n\}$, $k<n$ is a positive integer and the underlying braid of $[\alpha,k]$ is $\sigma_k$. The relations are the same as in the planar braid group, with the addition of the labelling, namely:

\begin{itemize} 
\item $ [\alpha,k] [\alpha \tau_k, k] =1 $;
\item $[\alpha, k_1] [\alpha \tau_{k_1}, k_2] = [\alpha, k_2]  [\alpha \tau_{k_2}, k_1]$,
where  $k_1+1< k_2<n$.
\end{itemize}
The generator $[\alpha,k]$ defines the crossing type 
$$(I_1^\alpha, I_2^\alpha, I_3^\alpha) = (\alpha^{-1}\{1,\ldots, k-1\}, \alpha^{-1}\{k, k+1\}, \alpha^{-1}\{k+2,\ldots, n\}).$$
The value of $\alpha$ defines the total order on each of the $I_r^\alpha$. 
Note that $[\alpha, k]$ is specified uniquely by the crossing type $(I_1^\alpha, I_2^\alpha, I_3^\alpha)$ together with a total order on each of the $I_r^\alpha$.  

We say that $[\alpha,k]$ is essential if $(I_1^\alpha, I_2^\alpha, I_3^\alpha)$ is; otherwise we call it non-essential. 

\subsection{Braids without essential crossings}

The groupoid $\PBgr_n$ has a distinguished subgroupoid $\PBgrness_n$ which consists of the planar braids that can be drawn without essential crossings; that is, without crossings whose crossing type is essential.
\begin{lem}\label{nocross}
For any pair of permutations $\alpha$ and $\beta$, the set $\mathrm{Hom}_{\PBgrness_n}(\alpha,\beta)$ consists of one element.
\end{lem}

\begin{proof}
If a braid has no essential crossings, there are no crossings at all to the left of the strand with label $n$. Therefore, the braid $b$ can be drawn as follows:
$$
\includegraphics[width=50pt]{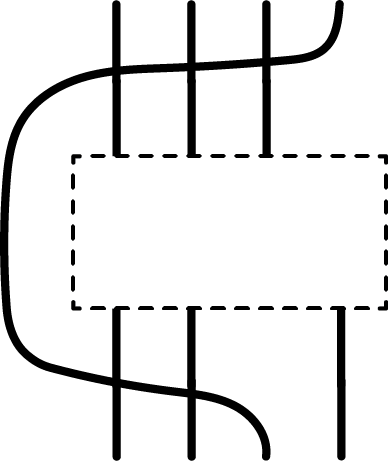}
$$
Removing the strand $n$ we obtain a braid with $n-1$ strands and no essential  crossings. Hence, it is sufficient to verify the Lemma for the planar braids on 3 strands for which it is immediate.
\end{proof}

\subsection{The Reidemeister-Schreier presentation for $\PP_n$}
Theorem~\ref{Higgins} implies the following
\begin{thm}\label{PnO}
The group $\PP_n$ has a presentation with the generators $[\alpha,k]$ and the relations 
\begin{itemize} 
\item[(i)]  $[\alpha,k]=1$ for non-essential $[\alpha,k]$.
\item[(ii)] $ [\alpha,k]^{-1} = [\alpha \tau_k, k]$;
\item[(iii)] $[\alpha, k] [\alpha \tau_{k}, l] = [\alpha, l] [\alpha \tau_{l}, k]$,
where  $k+1< l<n$;
\end{itemize}
\end{thm}

\begin{proof} 
By Theorem~\ref{Higgins}, a presentation for $\PP_n$ can be obtained from the presentation of the groupoid of planar braids $\PBgr_n$ by adding the relations that correspond to the edges of a spanning tree for the free groupoid on the generators $[\alpha,k]$ of $\PBgr_n$.

Recall that the permutohedron of order $n$ is a polytope whose vertices are the permutations of $\{1,\ldots, n\}$; two vertices are connected by an edge whenever the corresponding permutations differ by a right multiplication by an elementary transposition. 
The set of all the generators $[\alpha,k]$ can be identified with the set of oriented edges of the permutohedron of order $n$; namely, $[\alpha,k]$ is the edge connecting  $\alpha$ with $\alpha \tau_{k}$. A spanning tree for the 1-skeleton of the permutohedron lifts to a spanning tree for the free groupoid on the $[\alpha,k]$. 

There is a number of known spanning trees for the permutohedron, see \cite[Section~3.1]{HRW}. We consider the spanning tree that we call \emph{descending}. It is a rooted tree whose root is the permutation $(n, n-1, \ldots, 1)$. The parent of any other permutation $\alpha$ is $\alpha\tau_i$, where $i$ is defined as the number such that $\alpha^{-1}$ is decreasing on the set $\{i+1,\ldots, n\}$ and fails to be so on $\{i,\ldots, n\}$. In other words, if a permutation $\alpha$ is considered as a sequence of numbers $(\alpha^{-1}(1), \ldots, \alpha^{-1}(n))$, the edge of the tree towards the root is defined by exchanging the rightmost adjacent two numbers which fail to be in decreasing order.

A generator $[\alpha,k]$ that corresponds to an edge of the descending spanning tree is always non-essential. Indeed, by construction, $\alpha^{-1}$ is a decreasing function on the interval from $k+1$ to $m$. On the other hand, $\alpha^{-1}(k+1)$ lies in $I_2^\alpha$.

As a consequence, the presentation for $\PP_n$ can be obtained from that of $\PBgr_n$, by adding the relations $[\alpha,k]=1$ for those $[\alpha,k]$ which are edges of the descending spanning tree. In fact, we have $[\alpha,k]=1$ for \emph{all} non-essential $[\alpha,k]$. Indeed, the pure braid defined by a non-essential $[\alpha,k]$ as in Section \ref{braidgen} has no essential crossings and is, therefore, trivial by Lemma~\ref{nocross}.
\end{proof}

\section{The presentations in terms of crossing types}

\subsection{Decreasing generators}\label{gen}

Call a generator $[\alpha,k]$ \emph{decreasing} if $\alpha$ decreases on each of the $I_r^\alpha$. Note that the decreasing generators are in one-to-one correspondence with crossing types.

\begin{lem}\label{ithree} Two generators which give rise to the same crossing type with the same order on $I_1$ and on $I_2$ define the same element of $\PP_n$.
\end{lem}
In particular, we can always assume that $\alpha$ is decreasing on $I_3^\alpha$.
\begin{proof} 
Fix an essential generator $[\beta, k]$ and consider the set $S$ of all the generators $[\alpha, k]$ such that
$$I_1^\alpha = I_1^\beta,\quad I_2^\alpha=  I_2^\beta,$$
and such that $$\alpha(i)=\beta(i)$$ for all $i\in I_1^\beta \cup I_2^\beta$. In other words, $[\alpha, k]$ defines the same crossing type, with  the same order on $I_1$ and $I_2$, as $[\beta, k]$. 
The set $S$ can be identified with the set of permutations of the set $I_3^\beta$. 

Consider the permutohedron whose vertices are labelled with elements of $S$. Every edge of this permutohedron corresponds to a relation of type (iii) in which the generators $[\alpha, k]$ and $[\alpha\tau_l, k]$ define the same crossing type as $[\beta, k]$, with the same order on $I_1$. Consider the descending spanning tree for the 1-skeleton of this permutohedron. An edge of this tree corresponds to a relation of type (iii) in which $[\alpha, l]$ and 
 $[\alpha\tau_k, l]$ are non-essential; with the help of (i) such relation can be written as 
 $$[\alpha, k] = [\alpha\tau_l, k].$$
Since any two vertices of the permutohedron are connected to each other by a path lying in the spanning tree, we have 
$$[\alpha, k] = [\alpha', k]$$
for all $[\alpha, k], [\alpha', k] \in S$.
\end{proof}

\begin{cor}\label{shtrix}The relations (iii) in the presentation of Theorem~\ref{PnO} can be replaced by the following relations:
\begin{itemize} 
\item[(iii$_1$)] $[\alpha \tau_{k}, l] = [\alpha, k]^{-1} [\alpha, l] [\alpha, k]$;
\item[(iii$_2$)] $[\alpha, k] = [\alpha\tau_l, k]$,
\end{itemize}
where  $k+1< l<n$.
\end{cor}

This corollary can be used to define an algorithm that rewrites an arbitrary generator $[\alpha, k]$ as a word in the decreasing generators. 

\subsection{The rewriting algorithm}\label{algo}

Suppose that for each $i<k$ and for each $\alpha$ there is a word $w(\alpha, i)$ in the decreasing generators such that  
$$[\alpha, i] = w(\alpha, i)$$
as elements of $\PP_n$. Consider a generator $[\alpha, k]$ such that  $\alpha$ is decreasing on $I_2^\alpha$. By (iii$_2$), we can also assume that $\alpha$ is decreasing on $I_3^\alpha$.

Set $j=1$,  $\alpha_1=\alpha$ and $w_1=1$, the empty word in the decreasing generators.

\medskip

($\bullet$) If $[\alpha_j, k]$ is decreasing on $I_1^\alpha$, set $$w(\alpha, k)= w_j [\alpha_j, k] w_j^{-1}$$ and the algorithm terminates with the output $w(\alpha, k)$.

Otherwise, let $m<k-1$ be the greatest integer such that $\alpha_j^{-1}(m) < \alpha_j^{-1}(m+1)$. 
Then, we set  
$$w_{j+1}=w_j w(\alpha_j, m)$$
and
$$\alpha_{j+1}=\alpha_j \tau_{m},$$
increase the value of $j$ by 1 and go back to ($\bullet$).
\medskip 
 
We claim that the above algorithm eventually terminates and its output, the word $w(\alpha, k)$, is a word in decreasing generators that is equal to $[\alpha,k]$ as an element of $\PP_n$.

Indeed, for each $j$, we have
 $$\{\alpha_j^{-1}(1), \ldots, \alpha_j^{-1}(k-1)\} = \{\alpha^{-1}(1), \ldots, \alpha^{-1}(k-1)\}.$$
Passing from  $\alpha_j$ to  $\alpha_{j+1}$, we interchange the rightmost pair of numbers in the sequence $$(\alpha_j^{-1}(1), \ldots, \alpha_j^{-1}(k-1))$$ that fail to be in the decreasing order.  Such sequences for all $j$ may be considered as vertices of the permutohedron labelled with the permutations of  $I_1^{\alpha}$, and $\alpha_{j+1}$ is the parent vertex of $\alpha_j$ on the decreasing spanning tree. Moving from a vertex to its parent  we eventually arrive to the root of the tree, which is labelled with the decreasing permutation of $I_1^{\alpha}$. Each $\alpha_j$ is decreasing on $I_2^{\alpha}$ and $I_3^{\alpha}$; therefore, for the terminal value of $j$, the generator $[\alpha_j, k]$ is decreasing.

At each step, $w(\alpha_j, m)$ is a well-defined word in decreasing generators since $m<k$. Moreover, since $w(\alpha_j, m)=[\alpha_j,m]$ for $m<k$, 
(iii$_1$) implies that 
$$ [\alpha_j, k] = w(\alpha_j, m) [\alpha_{j+1}, k] w(\alpha_j, m)^{-1}$$
as elements of $\PP_n$ for all $j$. It follows that for all $j$ and, in particular, for the terminal value of $j$, we have
$$ [\alpha, k] = w_j [\alpha_{j}, k] w_j^{-1}.$$

In the algorithm above we have assumed that the generator $[\alpha, k]$ has the property that $\alpha$ decreases in $I_2^\alpha$. If $\alpha$ increases on 
$I_2^\alpha$, $\alpha\tau_k$ decreases on the same set. On the other hand, we have, by (ii), that $[\alpha\tau_k, k] =[\alpha, k]^{-1}$ so in order to produce the word in decreasing generators for $[\alpha, k]$ we take the inverse of the output of the algorithm for $[\alpha\tau_k, k]$.

We have proved

\begin{lem}\label{conjugation} For any generator $[\alpha, k]$ there exists a word $w$ in decreasing generators $[\beta, m]$ with $m<k-1$ such that
$$[\alpha, k] = w^{-1}[\alpha', k] w,$$
where $[\alpha', k]$ is a decreasing generator with the same crossing type as $[\alpha, k]$.
\end{lem}

\begin{cor}
The essential decreasing generators $[\alpha, k]$ form a set of generators for $\PP_n$.
\end{cor}

The statement of Lemma~\ref{conjugation} can be somewhat sharpened with the help of the following observation. Each step of the rewriting algorithm exchanges two adjacent numbers in the sequence $(\alpha^{-1}(1), \ldots, \alpha^{-1}(k-1))$. Note that if two numbers appear in this sequence in the decreasing order, they will never be exchanged. 
This implies 

\begin{lem}\label{w} 
In the notations of Lemma~\ref{conjugation} is can be assumed that for any pair of
numbers $p>q$ that appear in the decreasing order in the sequence $(\alpha^{-1}(1), \ldots, \alpha^{-1}(n))$, the word $w$ does not involve any generators whose crossing type $I$ satisfies 
either of the two conditions:
\begin{itemize}
\item $I_2=\{p,q\}$;
\item $p\in I_3$ and $q\in I_2$.
\end{itemize}
\end{lem}

\subsection{Relations} 
A relation of type (iii) is specified by a triple $[\alpha, k, l]$ with $k+1< l<n$, or, what is the same, by a double crossing type 
$$(J_1^\alpha,\ldots, J_5^\alpha) =(\alpha^{-1}\{1,\ldots,k-1\}, \alpha^{-1}\{k, k+1\}, \alpha^{-1}\{k+2,\ldots, l-1\}, \alpha^{-1}\{l, l+1\},
\alpha^{-1}\{l+2,\ldots, n\}),$$ together with an order on each $J_r^\alpha$.

Let us call a triple $[\alpha, k, l]$ \emph{decreasing} if $\alpha$ is decreasing on $J_r^\alpha$ for each $r$.

\begin{lem} 
In the presentation given by Theorem~\ref{PnO} and Corollary~\ref{shtrix}, the set of all 
relations of type (iii$_1$) can be replaced by the subset indexed by the decreasing triples only.
\end{lem}

\begin{proof} 

Let $R_i$ be the set of relations (iii$_1$) with $|J_3^\alpha|\leq i$ together with all the relations (iii$_2$). We will use induction on $i$ to show that $R_i$ can be replaced by the set of the relations from $R_i$ indexed by decreasing triples only. This will imply the statement of the Lemma since $R_{i+1}=R_i$ for $i\geq n-4$.

The triples $[\alpha,k,l]$ and $[\alpha\tau_m,k,l]$ where $\{m,m+1\}\subseteq J_r^\alpha$, $r\in \{1,2,4,5\}$, produce equivalent relations modulo (ii) and (iii$_2$). 
This is immediately clear for $J_2^\alpha$, $J_4^\alpha$ and $J_5^\alpha$. As for $J_1^\alpha$, transposing two adjacent elements in positions $m$ and $m+1$ results in the conjugation of the corresponding relation by $[\alpha,m]$.
This settles the case $i=1$. 

Now, consider the case of arbitrary $i$. By the previous argument, we can assume that $\alpha$ is decreasing on $J_r^\alpha$ for $r\neq 3$.  Take $m$ such that $k+1 < m < l-1$. 
The relations corresponding to $[\alpha,k,l]$ and $[\alpha\tau_m,k,l]$
are
\begin{equation}\label{eone}
[\alpha \tau_{k}, l] = [\alpha, k]^{-1}[\alpha, l] [\alpha, k]
\end{equation}
and
\begin{equation}\label{etwo}
[\alpha\tau_{m}\tau_{k}, l] = [\alpha\tau_{m}, k]^{-1}[\alpha\tau_{m}, l] [\alpha\tau_{m}, k],
\end{equation}
respectively.
The relation 
$$ [\alpha \tau_{m}, l] = [\alpha, m]^{-1}[\alpha, l] [\alpha, m]$$
belongs to $R_{i-1}$ so
$$ [\alpha \tau_{m}\tau_{k}, l] = ([\alpha, m] [\alpha\tau_{m}, k])^{-1}\cdot[\alpha, l]  \cdot [\alpha, m] [\alpha\tau_{m}, k]$$
is equivalent to (\ref{etwo}) modulo $R_{i-1}$.
Similarly, 
$$ [\alpha\tau_{k} \tau_{m}, l] = [\alpha\tau_{k}, m]^{-1}[\alpha\tau_{k}, l] [\alpha\tau_{k}, m]$$
lies in $R_{i-1}$ so, modulo $R_{i-1}$, we have that 
$$[\alpha\tau_{k}, l] = [\alpha\tau_{k}, m][\alpha\tau_{k} \tau_{m}, l] [\alpha\tau_{k}, m]^{-1} = 
[\alpha\tau_{k}, m]
[\alpha\tau_{m}, k]^{-1}
[\alpha, m]^{-1}\cdot
[\alpha, l]  \cdot 
[\alpha, m] [\alpha\tau_{m}, k]
[\alpha\tau_{k}, m]^{-1}$$
is equivalent to  (\ref{etwo}). On the other hand, $$[\alpha, m] [\alpha\tau_{m}, k]
[\alpha\tau_{k}, m]^{-1}=[\alpha, k],$$ since $[\alpha\tau_{m}, k]=[\alpha, k]$ by (iii$_2$) and this relation belongs to $R_{i-1}$. Therefore, modulo $R_{i-1}$,  (\ref{eone}) is equivalent to  (\ref{etwo}).
\end{proof}

\begin{proof}[Proof of Theorem~\ref{presentation}]
The rewriting algorithm gives a unique expression for any essential generator $[\alpha,k]$ in terms of the decreasing essential generators.
In order to obtain a presentation  for $\PP_n$ in terms of  these generators, one should substitute these expressions into the relations (iii) 
coming from the decreasing triples $[\alpha, k, l]$.

Such a triple gives the relation 
$$[\alpha, k][\alpha\tau_{k}, l] = [\alpha, l] [\alpha, k].$$
Let $[\alpha', k]$ and $[\alpha'', l]$ be the decreasing generators with the same crossing types as $[\alpha, k]$ and $[\alpha, l]$, respectively. Then, we have
$$[\alpha, k]=[\alpha', k]$$
in $\PP_n$ by (iii$_2$), and Lemma~\ref{conjugation} gives 
$$[\alpha, l]=u^{-1}[\alpha'',l]u$$ and $$[\alpha\tau_{k}, l]=v^{-1}[\alpha'',l]v,$$ where $u,v\in\PP_n$ are words in decreasing generators obtained with the help of the descending tree. Therefore,
\begin{equation}\label{finrel}
(u[\alpha', k] v^{-1})\cdot[\alpha'',l] = [\alpha'',l]\cdot (u [\alpha', k]v^{-1}).
\end{equation}

If the double crossing type of $[\alpha, k,l]$ is non-essential, (\ref{finrel}) holds trivially. Indeed, if $[\alpha, k,l]$ is non-essential, either there exists $y\in J_4$ such that $y>x$ for all $x\in J_5$, or there exists $y\in J_2$ such that $y>x$ for all $x\in J_3$. In the first case, the generator $[\alpha'',l]$ is non-essential, and, hence, trivial. In the second case, it follows from the rewriting algorithm (or, in other words, from the definition of the descending spanning tree) that $u [\alpha', k]=v$ as reduced words in the decreasing generators.

Now, if the double crossing type $J$ of a decreasing triple $[\alpha, k,l]$ is essential, take $$u_J=u, \quad \text{and}\quad v_J=v,$$  where $J$ is the double crossing type of $[\alpha, k,l]$.
 Then, (\ref{finrel}) produces the relations described in Theorem~\ref{presentation}. 

Let us now see that neither $u$ nor $v$ contain the generator $[\alpha', k]$. In Lemma~\ref{w}, take $$p=\alpha^{-1}(k),\quad  q=\alpha^{-1}(k+1),\quad [\beta,m]=[\alpha, l], \quad [\beta',m]=[\alpha'', l].$$ Since
$[\alpha,k,l]$ is decreasing, $p>q$ and, by the first part of Lemma~\ref{w}, $u$ cannot contain $[\alpha', k]$. 

The word $v$ cannot contain $[\alpha', k]$ either. We have, by construction, $$v= v'\cdot [\alpha\tau_{k}, k+1]^{-1}.$$ Indeed, $\alpha^{-1}$ (and, hence, $(\alpha\tau_{k})^{-1}$) is decreasing on the interval $k+1< x < l$ and, since $[\alpha,k,l]$ is essential, $$(\alpha\tau_{k})^{-1}(k+1)<(\alpha\tau_{k})^{-1}(k+2).$$ Now, apply the second part of Lemma~\ref{w} with: 
 \smallskip
 
 \begin{itemize}
 \item $p=(\alpha\tau_{k})^{-1}(k+2)$ and $q=(\alpha\tau_{k})^{-1}(k+1)$;
 \item $[\beta,m]=[\alpha\tau_{k}\tau_{k+1}, l]$
 \item $[\beta', m]=[\alpha'', l]$.
 \end{itemize}
 \smallskip
  
\noindent Then we have $$[\beta,m]=v'^{-1}\cdot [\beta', m]\cdot v'$$ and, therefore, $v'$ cannot contain $[\alpha', k]$, since for $[\alpha', k]$ we have $(\alpha\tau_{k})^{-1}(k+2)\in I_3$ and $(\alpha\tau_{k})^{-1}(k+1)\in I_2$.

\end{proof}

\subsection{The groups $\PP_n$ with $n\leq 6$.} There are no essential double crossing types on fewer than 6 strands and, therefore, $\PP_n$ with $n\leq 5$ are free. It is not hard to enumerate the essential crossing types for 3, 4 and 5 strands; there are 1, 7 and 31 of them, respectively, as expected. 

For $n=6$, there are 111 essential crossing types and  20 essential double crossing types. Each essential double crossing type $J$ is determined uniquely by $J_{[2]}$. Indeed, if $J$ is essential, $J_1=\emptyset$ and $|J_3|=|J_5|=1$ and the only element of $J_3$ is greater than any of the two elements of $J_2$. Note that, in contrast, $J_{[1]}$ does not determine $J$.

Let us denote an essential crossing type $I=(I_1, I_2, I_3)$ by  listing first the elements of $I_1$ and then, in angular brackets, the elements of $I_2$, in some order. Since the number of strands is fixed, omitting $I_3$ creates no ambiguity. For an essential double crossing type $J$ on 6 strands, assume that $J_2=\{a_1, a_2\}$ and $J_3=\{a_3\}$ (so that, for instance, $J_{[1]}=\langle a_2 a_1\rangle$).
The word $u_J J_{[1]} v_J^{-1}$, by the construction of the previous subsection, is equal to 
$$\langle a_3 a_2\rangle\cdot a_2\langle a_3 a_1\rangle\cdot \langle a_2 a_1\rangle\cdot a_1\langle a_3 a_2\rangle^{-1}\cdot \langle a_3 a_1\rangle^{-1}\cdot a_3\langle a_2 a_1\rangle^{-1}.$$
This word can be chosen as a new generator instead of one of the crossing types: if $a_3\neq 6$, we replace the generator $a_3\langle a_2 a_1\rangle$ with it; otherwise, we replace $\langle a_2 a_1\rangle$.
Then, the only relations in $\PP_6$ are 20 commutators of distinct pairs of generators and we see that $$\PP_6 = F_{71}*(F_2^{\mathrm{ab}})^{*20}.$$

For $n>6$, the task of simplifying the presentation in terms of crossing types is less straightforward. In particular, it is not clear whether  $\PP_7$ is a right-angled Artin group.

\begin{rem}\label{explicit} As we note in the Section~\ref{ba}, our presentation is directly related to a certain known set of multiplicative generators in cohomology. However, the actual planar braids that correspond to the crossing types are rather complicated. 

In order to represent  a crossing type $I$ with a  planar braid, one can consider the decreasing essential generator $[\alpha, k]$ corresponding to $I$ as a generator of the planar braid groupoid $\PBgr_n$ and produce a pure planar braid with the help of the descending spanning tree as explained in Section~\ref{rs}. In other words, if we consider $[\alpha, k]$ as a braid with labelled strands, the planar pure braid it produces is $$t(\alpha)\cdot [\alpha, k] \cdot t(\alpha\tau_k)^{-1},$$ 
where $t(\beta)$ is the braid without essential crossings that connects the identity permutation to the permutation $\beta$.
\begin{figure}[ht]
$$\includegraphics[height=1.8in]{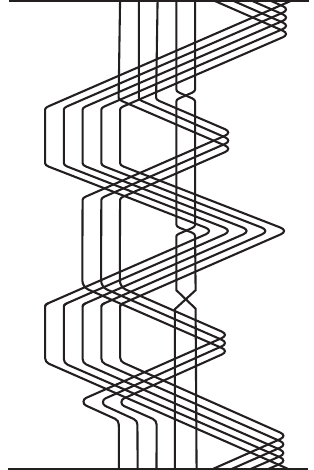}$$
\caption{The planar pure braid that corresponds to the crossing type $(\{1,2,3\},\{4,5\},\{6,7,8,9,10\})$.}
\label{genbr}

\end{figure}

\end{rem}
\subsection{Relation to Baryshnikov's generators in cohomology}\label{ba}
Each crossing type on $n$ strands defines a homomorphism $c_I: \PP_n\to \mathbb{Z}$ by assigning to a planar pure braid the number of its crossings of type $I$, taken with appropriate signs. It was shown by Baryshnikov \cite{B} that the homomorphisms defined by the essential crossing types generate $H^{*}(\PP_n, \mathbb{Z})$ as an algebra. 

The braid $b_I$ that represents an essential crossing type $I$ as in Remark~\ref{explicit} has precisely one essential crossing and the type of this crossing is $I$. In this sense, our generators are dual to Baryshnikov's generators of the cohomology since
$$c_I (b_{I'}) = \delta_{I,I'}.$$

\section{A presentation for the pure cactus group}.

Recall that the \emph{cactus group} $J_n$  has a presentation with the generators $s_{p,q}$, where $1\leq p< q\leq n$, and the following relations:
\begin{equation}\label{eqcac}
\begin{array}{rcll}
s_{p,q}^2&=& 1,&\\
s_{p,q}s_{m,r}&=& s_{m,r}s_{p,q} &\quad \text{if\ } [p,q]\cap [m,r] =\emptyset,\\
s_{p,q}s_{m,r}&=& s_{p+q-r, p+q-m}s_{p,q} &\quad \text{if\ }  [m,r] \subset [p,q].
\end{array}
\end{equation}
There is a homomorphism $J_n\to \Sigma_n$ to the symmetric group: it sends $s_{p,q}$ into the permutation $\tau_{p,q}$ of $\{1, \ldots, n\}$ which reverses the order of $p, p+1,\ldots, q$ and leaves the rest of the elements unchanged. The \emph{pure cactus group} $\Gamma_{n+1}$ is the kernel of this homomorphism. The groups $\Gamma_{i}$ are trivial for $i<4$, $\Gamma_{4}$ is an infinite cyclic group and $\Gamma_{5}$ is known to be the fundamental group of a connected sum of 5 real projective planes.

Elements of  $J_n$ can be depicted as planar braids of a more general kind, whose strands may have $k$-fold intersections for any $k\leq n$, see \cite{M}. The generators and relations for $J_n$ in terms of such pictures are shown in Figure~\ref{cactgen}. The planar braid group $\PB_n$ naturally sits inside $J_n$ as the subgroup of braids which have double intersections only \cite{BCL}.

\begin{figure}[ht]\label{cactgen}
$$\includegraphics[width=5.5in]{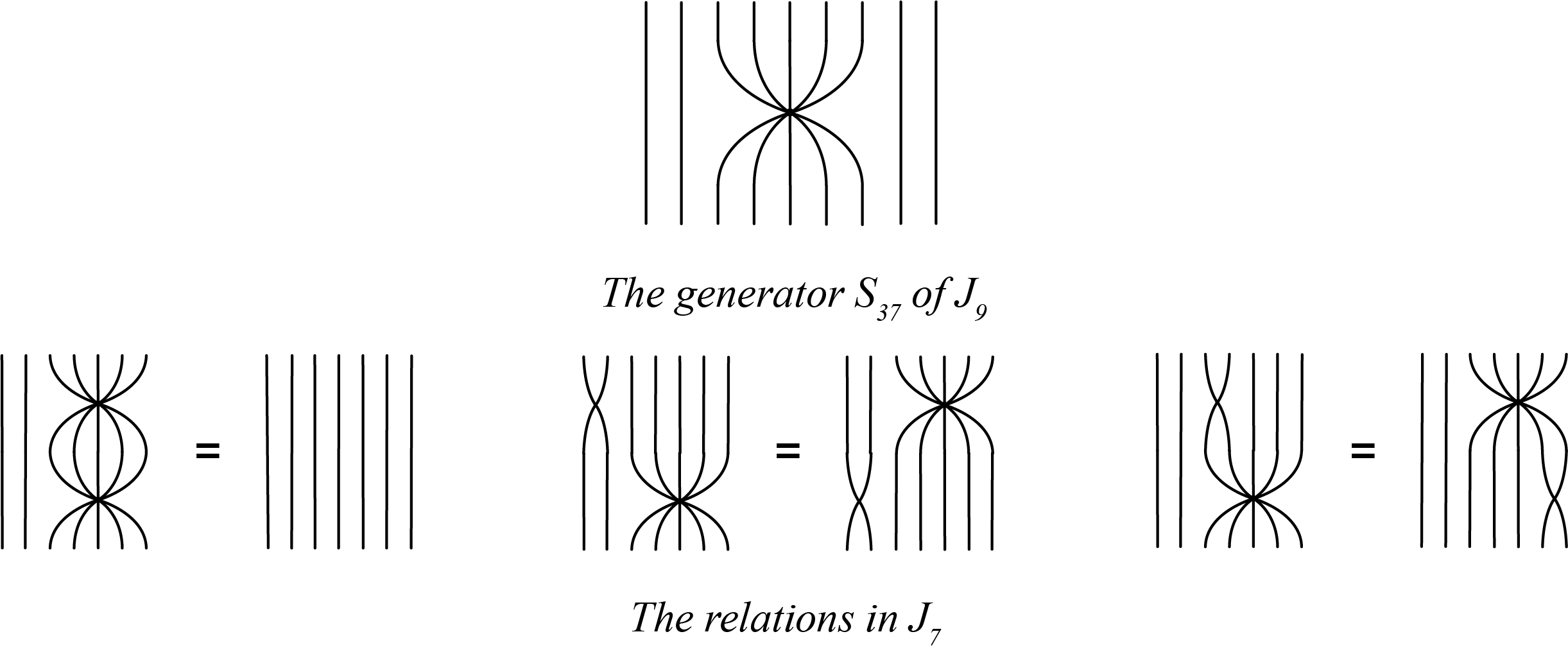}$$
\caption{Examples of generators and relations in cactus groups}
\end{figure}

The \emph{cactus groupoid} $\Jgr_n$ has the permutations of the set $\{1,\ldots, n\}$ as objects and $\mathrm{Hom}(\alpha, \beta)$ 
can be identified with the set of all $g\in J_n$ such that $\beta=\alpha\tau$ where $\tau$ is the permutation defined by $g$. This is exactly the same construction as that of the planar braid groupoid with the planar braids replaced by planar braids with higher-order intersections. It is generated by all pairs $[\alpha, p,q]\in \mathrm{Hom}(\alpha, \alpha\tau_{p,q})$, where $1\leq p<q\leq n$, subject to the relations
\begin{itemize} 
\item $ [\alpha,p,q] [\alpha \tau_{p,q}, p,q] =1 $;
\item $[\alpha,p,q] [\alpha \tau_{p,q}, m, r] =  [\alpha,m,r] [\alpha \tau_{m,r}, p, q]     \quad \text{if\ } [p,q]\cap [m,r] =\emptyset,$;
\item $[\alpha,p,q] [\alpha \tau_{p,q}, m, r]  = [\alpha, p+q-r, p+q-m]  [\alpha \tau_{p+q-r, p+q-m}, p,q] \quad \text{if\ }  [m,r] \subset [p,q].$
\end{itemize}

The generators $[\alpha,p,p+1]$ are precisely the generators of the groupoid $\PBgr_n$ and we can take the same spanning tree for the groupoid $\Jgr_n$  as for $\PBgr_n$. In view of the discussion in Section~\ref{RSp}, in order to obtain a presentation for the pure cactus group $\Gamma_{n+1}$, we only need to add to the above presentation of $\PBgr_n$ the relations 
$$[\alpha, p,p+1]=1, \quad\text{where\ } [\alpha,p] \text{\ is non-essential}.$$

This presentation is very far from being minimal, of course. For instance, while $\Gamma_4$ is infinite cyclic, our construction produces 7 generators. In general, the number of generators of $\Gamma_{n+1}$ is bounded below by $2^n - n(n + 1)/2 - 1$, the number of generators of $H_1(\Gamma_{n+1},\Z/2\Z)$ (see \cite{M}). The number of generators in our presentation, in contrast, has the order of  $(n+2)!/2$.

\end{document}